\documentclass[12pt]{article}
\usepackage{amsmath,amsthm,amssymb}
\usepackage{amssymb,latexsym}
\usepackage{mathrsfs}
\newtheorem{theorem}{Theorem}

\newtheorem{corollary}{Corollary}
\newtheorem{lemma}{Lemma}

\begin{document}

\title{\bf A Diophantine inequality with five squares of Piatetski-Shapiro primes}

\author{\bf S. I. Dimitrov}

\date{}

\maketitle

\begin{abstract}
Let $[\,\cdot\,]$  denote the floor function.
Assume that $\lambda_1, \lambda_2, \lambda_3, \lambda_4, \lambda_5$ are nonzero real numbers, not all of the same sign, that $\lambda_1/\lambda_2$ is irrational, and that $\eta$ is a real number. 
Let $\frac{71}{72}<\gamma<1$ and $\theta>0$.
We prove that there exist infinitely many quintuples of primes $p_1,\, p_2,\, p_3,\, p_4,\, p_5$ satisfying the Diophantine inequality
\begin{equation*}
\big|\lambda_1p^2_1 + \lambda_2p^2_2 + \lambda_3p^2_3+ \lambda_4p^2_4 + \lambda_5p^2_5+\eta\big|<\big(\max p_j\big)^{\frac{71-72\gamma}{29}+\theta}\,,
\end{equation*}
where $p_i=[n_i^{1/\gamma}]$, $i=1,\,2,\,3,\,4,\,5$. We also prove analogous theorems by raising the last variable in the inequality to the third and fourth powers.
\quad\\
{\bf  2020 Math.\ Subject Classification}: 11D75  $\cdot$  11P05
\end{abstract}

\section{Introduction and statement of the result}
\indent

In 1938, Hua \cite{Hua1} proved that every sufficiently large positive integer $N\equiv 5 \,(\textmd{mod}\; 24)$ can be presented in the form
\begin{equation}\label{Hua}
N=p^2_1+p^2_2+p^2_3+p^2_4+p^2_5\,,
\end{equation}
where $p_1,\, p_2,\, p_3,\, p_4,\, p_5$ are prime numbers. In 2004, Harman \cite{Harman} considered an analogous problem for Diophantine inequality.
He showed that there are infinitely many quintuples of primes $p_1,\, p_2,\, p_3,\, p_4,\, p_5$ such that
\begin{equation}\label{Harman}
\big|\lambda_1p^2_1 + \lambda_2p^2_2 + \lambda_3p^2_3+ \lambda_4p^2_4 + \lambda_5p^2_5+\eta\big|<\big(\max p_j\big)^{-\frac{1}{8}+\delta}\,,
\end{equation}
where $\lambda_1, \lambda_2, \lambda_3, \lambda_4, \lambda_5$ are non-zero real numbers, not all of the same sign, $\lambda_1/\lambda_2$ is irrational, $\eta$ is a real number and $\delta>0$.
Inequality \eqref{Harman} was transformed by raising the last variable to a power $k\geq3$ in papers \cite{Ge}, \cite{Li-Wang}, \cite{Liu}, \cite{Mu2016}, \cite{Mu2019}, \cite{Zhu} .

Another interesting problem involves the study of Diophantine equations and inequalities with primes of a special form.
In 1953, Piatetski-Shapiro \cite{Shapiro1953} proved that for any fixed $\frac{11}{12}<\gamma<1$, there are infinitely many primes of the form $p = [n^{1/\gamma}]$.
Such primes are called Piatetski-Shapiro primes of type $\gamma$. 
Afterwards, the admissible range of $\gamma$ was improved several times, with the best result currently known due to Rivat and Wu \cite{Rivat-Wu}, who established that $\frac{205}{243}<\gamma<1$.

In 1998, Zhai \cite{Zhai} studied equation \eqref{Hua} over the set of Piatetski–Shapiro primes.
More precisely, he proved that for any fixed $\frac{43}{44}<\gamma<1$ and  every sufficiently large positive integer $N\equiv 5 \,(\textmd{mod}\; 24)$, 
equation \eqref{Hua} is solvable with primes $p_i=[n_i^{1/\gamma}]$, $i=1,\,2,\,3,\,4,\,5$. 
Subsequently, Zhai's result was improved by Zhang and Zhai \cite{Zhang}, who established that $\frac{249}{256}<\gamma<1$.
We also note that Li and Zhang \cite{Li-Zhang} solved equation \eqref{Hua} with one prime of the form $[n^{1/\gamma}]$, for $\frac{205}{243}<\gamma<1$.

As a continuation of these studies, we consider inequality \eqref{Harman} with Piatetski-Shapiro primes.  
\begin{theorem}\label{Theorem1}
Suppose that  $\lambda_1, \lambda_2, \lambda_3, \lambda_4, \lambda_5$ are nonzero real numbers, not all of the same sign, that $\lambda_1/\lambda_2$ is irrational, and that $\eta$ is real. 
Let $\frac{71}{72}<\gamma<1$ and $\theta>0$.
Then there exist infinitely many ordered  quintuples of Piatetski-Shapiro primes $p_1,\, p_2,\, p_3,\, p_4,\, p_5$ of type $\gamma$ such that
\begin{equation*}
\big|\lambda_1p^2_1 + \lambda_2p^2_2 + \lambda_3p^2_3+ \lambda_4p^2_4 + \lambda_5p^2_5+\eta\big|<\big(\max p_j\big)^{\frac{71-72\gamma}{29}+\theta}\,.
\end{equation*}
\end{theorem}
Next, we prove analogous theorems by raising the last variable in the inequality to the third and fourth powers.
\begin{theorem}\label{Theorem2}
Suppose that  $\lambda_1, \lambda_2, \lambda_3, \lambda_4, \lambda_5$ are nonzero real numbers, not all of the same sign, that $\lambda_1/\lambda_2$ is irrational, and that $\eta$ is real. 
Let $\frac{129}{130}<\gamma<1$ and $\theta>0$.
Then there exist infinitely many ordered  quintuples of Piatetski-Shapiro primes $p_1,\, p_2,\, p_3,\, p_4,\, p_5$ of type $\gamma$ such that
\begin{equation*}
\big|\lambda_1p^2_1 + \lambda_2p^2_2 + \lambda_3p^2_3+ \lambda_4p^2_4 + \lambda_5p^3_5+\eta\big|<\big(\max p_j\big)^{\frac{129-130\gamma}{58}+\theta}\,.
\end{equation*}
\end{theorem}
\begin{theorem}\label{Theorem3}
Suppose that  $\lambda_1, \lambda_2, \lambda_3, \lambda_4, \lambda_5$ are nonzero real numbers, not all of the same sign, that $\lambda_1/\lambda_2$ is irrational, and that $\eta$ is real. 
Let $\frac{245}{246}<\gamma<1$ and $\theta>0$.
Then there exist infinitely many ordered  quintuples of Piatetski-Shapiro primes $p_1,\, p_2,\, p_3,\, p_4,\, p_5$ of type $\gamma$ such that
\begin{equation*}
\big|\lambda_1p^2_1 + \lambda_2p^2_2 + \lambda_3p^2_3+ \lambda_4p^2_4 + \lambda_5p^4_5+\eta\big|<\big(\max p_j\big)^{\frac{245-246\gamma}{116}+\theta}\,.
\end{equation*}
\end{theorem}

\section{Notations}
\indent

The letter $p$ will always denote a prime number. By $\delta$ we denote an arbitrarily small positive number, not necessarily the same in different occurrences.
As usual, $[t]$ denotes the integer part of $t$. Moreover $e(t)=e^{2\pi it}$.
Let $\gamma_1$, $\gamma_2$, $\gamma_3$, $\theta$ and $\lambda_0$ be a real constants such that $\frac{71}{72}<\gamma_2<1$, $\frac{129}{130}<\gamma_3<1$, $\frac{245}{246}<\gamma_4<1$, $\theta>0$ and $0<\lambda_0<1$.
Since $\lambda_1/\lambda_2$ is irrational, there are infinitely many different convergents $a_0/q_0$ to its continued fraction, with
\begin{equation}\label{lambda12a0q0}
\bigg|\frac{\lambda_1}{\lambda_2} - \frac{a_0}{q_0}\bigg|<\frac{1}{q_0^2}\,,\quad (a_0, q_0) = 1\,,\quad a_0\neq0
\end{equation}
and $q_0$ is arbitrary large. Denote
\begin{align}
\label{X}
&X=q_0^\frac{58}{27}\,;\\
\label{Delta}
&\Delta=X^{-\frac{27}{29}}\log X\,;\\
\label{varepsilon2}
&\varepsilon_2=X^{\frac{71-72\gamma_2}{58}+\theta}\,;\\
\label{varepsilon3}
&\varepsilon_3=X^{\frac{129-130\gamma_3}{116}+\theta}\,;\\
\label{varepsilon4}
&\varepsilon_4=X^{\frac{245-246\gamma_4}{232}+\theta}\,;\\
\label{H}
&H_k=\frac{\log^2X}{\varepsilon_k}\,;\\
\label{Sk}
&S_{k, i}(t)=\sum\limits_{\lambda_0X<p^k\leq X\atop{p=[n^{1/\gamma_i}]}}p^{1-\gamma_i}e(t p^k)\log p\,;\\
\label{Sigma}
&\Sigma_k(t)=\sum\limits_{\lambda_0X<p^k\leq X}e(t p^k)\log p\,;\\
\label{U}
&U_k(t)=\sum\limits_{\lambda_0X<n^k\leq X}e(t n^k)\,;\\
\label{Ik}
&I_k(t)=\int\limits_{(\lambda_0X)^\frac{1}{k}}^{X^\frac{1}{2}}e(t y^k)\,dy\,.
\end{align}

\section{Auxiliary lemmas}
\indent

\begin{lemma}\label{Fourier} Let $\varepsilon_k>0$ and $l\in \mathbb{N}$. There exists a function $\theta(y)$ which is $l$ times continuously differentiable and such that
\begin{align*}
&\theta(y)=1\hspace{12.5mm}\mbox{for }\hspace{5mm}|y|\leq 3\varepsilon_k/4\,;\\
&0<\theta(y)<1\hspace{5mm}\mbox{for}\hspace{7mm}3\varepsilon_k/4 <|y|< \varepsilon_k\,;\\
&\theta(y)=0\hspace{12.5mm}\mbox{for}\hspace{7mm}|y|\geq \varepsilon_k\,.
\end{align*}
and its Fourier transform
\begin{equation*}
\Theta(x)=\int\limits_{-\infty}^{\infty}\theta(y)e(-xy)dy
\end{equation*}
satisfies the inequality
\begin{equation*}
|\Theta(x)|\leq\min\bigg(\frac{7\varepsilon_k}{4},\frac{1}{\pi|x|},\frac{1}{\pi |x|}\bigg(\frac{l}{2\pi |x|\varepsilon_k/8}\bigg)^l\bigg)\,.
\end{equation*}
\end{lemma}
\begin{proof}
See (\cite{Shapiro1952}).
\end{proof}

\begin{lemma}\label{Shapiroasymp} For any fixed $\frac{2426}{2817}<\gamma<1$, we have
\begin{equation*}
\sum\limits_{p\leq X\atop{p=[n^{1/\gamma}]}}1\sim \frac{X^\gamma}{\log X}\,.
\end{equation*}
\end{lemma}
\begin{proof}
See (\cite{Rivat-Sargos}, Theorem 1).
\end{proof}

\begin{lemma}\label{S2asymptotic} Let $\frac{13}{14}<\gamma_2<1$. Then
\begin{equation*}
S_{2, 2}(t)=\gamma_2\Sigma_2(t)+\mathcal{O}\left(X^{\frac{21-7\gamma_2}{29}+\delta}\right)\,.
\end{equation*}
\end{lemma}
\begin{proof}
See (\cite{Dimitrov2025}, Lemma 9).
\end{proof}

\begin{lemma}\label{S3asymptotic} Let $\frac{47}{48}<\gamma_3<1$. Then
\begin{equation*}
S_{3, 3}(t)=\gamma_3\Sigma_3(t)+\mathcal{O}\left(X^{\frac{143-48\gamma_3}{288}+\delta}\right)\,.
\end{equation*}
\end{lemma}
\begin{proof}
See (\cite{Dimitrov2026a}, Lemma 10).
\end{proof}

\begin{lemma}\label{S4asymptotic} Let $\frac{99}{100}<\gamma_4<1$. Then
\begin{equation*}
S_{4, 4}(t)=\gamma_4\Sigma_4(t)+\mathcal{O}\left(X^{\frac{149-50\gamma_4}{398}+\delta}\right)\,.
\end{equation*}
\end{lemma}
\begin{proof}
See (\cite{Dimitrov2026b}, Lemma 9).
\end{proof}

\begin{lemma}\label{Languasco} Let $k\geq1$ and $1/2X\leq Y\leq 1/2X^{1-\frac{5}{6k}+\delta}$. Then there exists a positive constant $c_1(\delta)$, which does not depend on $k$, such that
\begin{equation*}
\int\limits_{-Y}^Y\big|\Sigma_k(t)-U_k(t)\big|^2\,dt\ll\frac{X^{\frac{2}{k}-2}\log^2X}{Y}+Y^2X+X^{\frac{2}{k}-1}\mathrm{exp}\Bigg(-c_1\bigg(\frac{\log X}{\log\log X}\bigg)^{1/3}\Bigg)   \,.
\end{equation*}
\end{lemma}
\begin{proof}
See (\cite{Gambini}, Lemma 1 and Lemma 2).
\end{proof}

\begin{lemma}\label{Ikest} We have
\begin{equation*}
I_k(t)\ll X^{\frac{1}{k}-1}\min\Big(X,\, |t|^{-1}\Big)\,.
\end{equation*}
\end{lemma}
\begin{proof}
See (\cite{Titchmarsh}, Lemma 4.2).
\end{proof}

\begin{lemma}\label{intS2} We have
\begin{equation*}
\int\limits_{-\Delta}^\Delta\big|S_{2, k}(t)\big|^2\,dt\ll  \Delta X^{1-\frac{\gamma_k}{2}}\log^2X+\log^3X\,.
\end{equation*}
\end{lemma}
\begin{proof}
It follows directly from the proof of (\cite{Kumchev}, Lemma 17) by letting $c=2$ and using $X^\frac{1}{2}$ instead of $X$ there.
\end{proof}

\begin{lemma}\label{intS4} We have
\begin{equation*}
\int\limits_{0}^1\big|S_{2, k}(t)\big|^4\,dt\ll X^{2-\gamma_k+\delta}\,.
\end{equation*}
\end{lemma}
\begin{proof}
See (\cite{Zhai}, (12)).
\end{proof}

\begin{lemma}\label{intS8} We have
\begin{equation*}
\int\limits_{0}^1|S_{3,3}(t)|^8\,dt\ll X^{\frac{8}{3}-\gamma_3+\delta}\,.
\end{equation*}
\end{lemma}
\begin{proof}
See (\cite{Long}, Lemma 2.2).
\end{proof}

\begin{lemma}\label{intS16} We have
\begin{equation*}
\int\limits_{0}^1\big|S_{4, 4}(t)\big|^{16}\,dt\ll X^{4-\gamma_4+\delta}\,.
\end{equation*}
\end{lemma}
\begin{proof}
See (\cite{Dimitrov2026b}, Corollary 1).
\end{proof}

\begin{lemma}\label{Expsumest}
Suppose that $t \in \mathbb{R}$,\, $a \in \mathbb{Z}$,\, $q\in \mathbb{N}$,\, $\big|t-\frac{a}{q}\big|\leq\frac{1}{q^2}$\,, $(a, q)=1$.
Then
\begin{equation*}
\sum\limits_{p\le N}e(t p^2)\log p\ll N^{1+\delta}\bigg(\frac{1}{q}+\frac{1}{N^\frac{1}{2}}+\frac{q}{N^2}\bigg)^\frac{1}{4}\,.
\end{equation*}
\end{lemma}
\begin{proof}
See (\cite{Ghosh}, Theorem 2).
\end{proof}

\section{Initial steps}
\indent

Consider the sum
\begin{equation}\label{Gamma}
\Gamma_k(X)=\sum\limits_{\lambda_0X<p^2_1,p^2_2,p^2_3,p^2_4,p^k_5\leq X\atop{p_i=[n^{1/\gamma_k}_i],\, i=1,2,3,4,5}}
\theta\big(\lambda_1p^2_1 + \lambda_2p^2_2 + \lambda_3p^2_3+ \lambda_4p^2_4 + \lambda_5p^k_5+\eta\big)\prod\limits_{j=1}^{5}p^{1-\gamma_k}_j\log p_j\,.
\end{equation}
Using the inverse Fourier transform for the function $\theta(x)$, we write
\begin{align*}
\Gamma_k(X)&=\sum\limits_{\lambda_0X<p^2_1,p^2_2,p^2_3,p^2_4,p^k_5\leq X\atop{p_i=[n^{1/\gamma_k}_i],\, i=1,2,3,4,5}}\prod\limits_{j=1}^{5}p^{1-\gamma_k}_j\log p_j\\\
&\times\int\limits_{-\infty}^{\infty}\Theta(t)e\big((\lambda_1p^2_1 + \lambda_2p^2_2 + \lambda_3p^2_3+ \lambda_4p^2_4 + \lambda_5p^k_5+\eta)t\big)\,dt\\
&=\int\limits_{-\infty}^{\infty}\Theta(t)S_{2, k}(\lambda_1t)S_{2, k}(\lambda_2t)S_{2, k}(\lambda_3t)S_{2, k}(\lambda_4t)S_{k, k}(\lambda_5t)e(\eta t)\,dt\,.
\end{align*}
We decompose $\Gamma_k(X)$  into three integrals 
\begin{equation}\label{Gammadecomp}
\Gamma_k(X)=A_k(X)+B_k(X)+C_k(X)\,,
\end{equation}
where

\begin{align}
\label{Ak}
&A_k(X)=\int\limits_{|t|<\Delta}\Theta(t)S_{2, k}(\lambda_1t)S_{2, k}(\lambda_2t)S_{2, k}(\lambda_3t)S_{2, k}(\lambda_4t)S_{k, k}(\lambda_5t)e(\eta t)\,dt\,,\\
\label{Bk}
&B_k(X)=\int\limits_{\Delta\leq|t|\leq H_k}\Theta(t)S_{2, k}(\lambda_1t)S_{2, k}(\lambda_2t)S_{2, k}(\lambda_3t)S_{2, k}(\lambda_4t)S_{k, k}(\lambda_5t)e(\eta t)\,dt\,,\\
\label{Ck}
&C_k(X)=\int\limits_{|t|>H_k}\Theta(t)S_{2, k}(\lambda_1t)S_{2, k}(\lambda_2t)S_{2, k}(\lambda_3t)S_{2, k}(\lambda_4t)S_{k, k}(\lambda_5t)e(\eta t)\,dt\,.
\end{align}

\section{Lower bound for $\mathbf{A_k(X)}$}
\indent

Put
\begin{equation}\label{JX}
J=\gamma_k^5\int\limits_{-\Delta}^\Delta\Theta(t)I_2(\lambda_1t)I_2(\lambda_2t)I_2(\lambda_3t)I_2(\lambda_4t)I_k(\lambda_5t)e(\eta t)\,dt\,.
\end{equation}
Now \eqref{Ak} and \eqref{JX} give us
\begin{align*}
A_k(X)-J
&=\gamma_k^4\int\limits_{-\Delta}^\Delta\Theta(t)\Big(S_{2, k}(\lambda_1t)-\gamma_k I_2(\lambda_1t)\Big)I_2(\lambda_2t)I_2(\lambda_3t)I_2(\lambda_4t)I_k(\lambda_5t)e(\eta t)\,dt\nonumber\\
&+\gamma_k^3\int\limits_{-\Delta}^\Delta\Theta(t)S_{2, k}(\lambda_1t)\Big(S_{2, k}(\lambda_2t)-\gamma_k I_2(\lambda_2t)\Big)I_2(\lambda_3t)I_2(\lambda_4t)I_k(\lambda_5t)e(\eta t)\,dt
\end{align*}
\begin{align}\label{Gamma1-JX}
&+\gamma_k^2\int\limits_{-\Delta}^\Delta\Theta(t)S_{2, k}(\lambda_1t)S_{2, k}(\lambda_2t)\Big(S_{2, k}(\lambda_3t)-\gamma_k I_2(\lambda_3t)\Big))I_2(\lambda_4t)I_k(\lambda_5t)e(\eta t)\,dt\nonumber\\
&+\gamma_k\int\limits_{-\Delta}^\Delta\Theta(t)S_{2, k}(\lambda_1t)S_{2, k}(\lambda_2t)S_{2, k}(\lambda_3t)\Big(S_{2, k}(\lambda_4t)-\gamma_k I_2(\lambda_4t)\Big))I_k(\lambda_5t)e(\eta t)\,dt\nonumber\\
&+\int\limits_{-\Delta}^\Delta\Theta(t)S_{2, k}(\lambda_1t)S_{2, k}(\lambda_2t)S_{2, k}(\lambda_3t)S_{2, k}(\lambda_4t)\Big(S_{k, k}(\lambda_5t)-\gamma_k I_k(\lambda_5t)\Big)e(\eta t)\,dt\nonumber\\
&=\gamma_k^4J_1+\gamma_k^3J_2+\gamma_k^2J_3+\gamma_k J_4+J_5\,,
\end{align}
say. In the following subsections, we estimate $J_i$ for $1\leq i\leq 5$ and $J$. Since the estimations for $J_2, J_3$ and $J_4$ are similar to those for $J_1$, we omit them.

\subsection{Upper bound for $\mathbf{J_1}$}
\indent

The Euler summation formula implies that
\begin{equation}\label{I-U}
I_j(t)-U_j(t)\ll1+|t|X\,, \quad j=2, k\,.
\end{equation}
From Lemma \ref{Fourier}, we have
\begin{align}\label{J1est}
J_1&\ll\varepsilon_k\int\limits_{-\Delta}^\Delta\big|S_{2, k}(\lambda_1t)-\gamma_k\Sigma_2(\lambda_1t)\big|\big|I_2(\lambda_2t)I_2(\lambda_3t)I_2(\lambda_4t)I_k(\lambda_5t)\big|\,dt\nonumber\\
&+\varepsilon_k\int\limits_{-\Delta}^\Delta\big|\gamma_k\Sigma_2(\lambda_1t)-\gamma_k I_2(\lambda_1t)\big|\big|I_2(\lambda_2t)I_2(\lambda_3t)I_2(\lambda_4t)I_k(\lambda_5t)\big|\,dt\nonumber\\
&\ll\varepsilon_k\int\limits_{-\Delta}^\Delta\big|S_{2, k}(\lambda_1t)-\gamma_k\Sigma_2(\lambda_1t)\big|\big|I_2(\lambda_2t)I_2(\lambda_3t)I_2(\lambda_4t)I_k(\lambda_5t)\big|\,dt\nonumber\\
&+\varepsilon_k\int\limits_{-\Delta}^\Delta\big|\Sigma_2(\lambda_1t)-U_2(\lambda_1t)\big|\big|I_2(\lambda_2t)I_2(\lambda_3t)I_2(\lambda_4t)I_k(\lambda_5t)\big|\,dt\nonumber\\
&+\varepsilon_k\int\limits_{-\Delta}^\Delta\big|U_2(\lambda_1t)-I_2(\lambda_1t)\big|\big|I_2(\lambda_2t)I_2(\lambda_3t)I_2(\lambda_4t)I_k(\lambda_5t)\big|\,dt\nonumber\\
&=\varepsilon_k\Big(J'_1+J''_1+J'''_1\Big)\,,
\end{align}
say. First, we estimate $J'_1$. Now Lemma \ref{S2asymptotic} and Lemma \ref{Ikest} yield

\begin{equation}\label{J'1est}
J'_1\ll X^{\frac{21-7\gamma_k}{29}+\delta}\int\limits_{-\Delta}^\Delta\frac{X^{\frac{1}{k}-\frac{5}{2}}}{|t|^4}\,dt\ll X^{\frac{1}{k}-\frac{14\gamma_k+103}{58}+\delta}\Delta^{-3}\,.
\end{equation}
Next, we estimate $J''_1$ and $J'''_1$. Using Cauchy's inequality, \eqref{I-U}, Lemma \ref{Languasco}, Lemma \ref{Ikest} and arguing as in (\cite{Mu2016}, p. 487), we obtain 
\begin{align}\label{J''1est}
J''_1
\ll\frac{X^{1+\frac{1}{k}}}{e^{(\log X)^{1/4}}}
\end{align}
and
\begin{align}\label{J'''1est}
J'''_1&\ll X^{\frac{1}{2}+\frac{1}{k}}\,.
\end{align}

\subsection{Upper bound for $\mathbf{J_5}$}
\indent

By Lemma \ref{Fourier}, we get
\begin{align}\label{J5est}
J_5&\ll\varepsilon_k\int\limits_{-\Delta}^\Delta\big|S_{2, k}(\lambda_1t)S_{2, k}(\lambda_2t)S_{2, k}(\lambda_3t)S_{2, k}(\lambda_4t)\big|\big|S_{k, k}(\lambda_5t)-\gamma_k\Sigma_k(\lambda_5t)\big|\,dt\nonumber\\
&+\varepsilon_k\int\limits_{-\Delta}^\Delta\big|S_{2, k}(\lambda_1t)S_{2, k}(\lambda_2t)S_{2, k}(\lambda_3t)S_{2, k}(\lambda_4t)\big| \big|\gamma_k\Sigma_k(\lambda_5t)-\gamma_k I_k(\lambda_5t)\big|\,dt\nonumber\\
&\ll\varepsilon_k\int\limits_{-\Delta}^\Delta\big|S_{2, k}(\lambda_1t)S_{2, k}(\lambda_2t)S_{2, k}(\lambda_3t)S_{2, k}(\lambda_4t)\big| \big|S_{k, k}(\lambda_5t)-\gamma_k\Sigma_k(\lambda_5t)\big|\,dt\nonumber\\
&+\varepsilon_k\int\limits_{-\Delta}^\Delta\big|S_{2, k}(\lambda_1t)S_{2, k}(\lambda_2t)S_{2, k}(\lambda_3t)S_{2, k}(\lambda_4t)\big| \big|\Sigma_k(\lambda_5t)-U_k(\lambda_5t)\big|\,dt \nonumber\\
&+\varepsilon_k\int\limits_{-\Delta}^\Delta\big|S_{2, k}(\lambda_1t)S_{2, k}(\lambda_2t)S_{2, k}(\lambda_3t)S_{2, k}(\lambda_4t))\big| \big|U_k(\lambda_5t)-I_k(\lambda_5t)\big|\,dt\nonumber\\
&=\varepsilon_k\Big(J'_5+J''_5+J'''_5\Big)\,,
\end{align}
say. First, we estimate $J'_5$ in three cases.

\smallskip

\textbf{Case 1.} $k=2$.

Using  Lemma \ref{Shapiroasymp}, Lemma \ref{S2asymptotic} and Lemma \ref{intS2}, we deduce 
\begin{equation}\label{J'5est1}
J'_5\ll X^{\frac{21-7\gamma_2}{29}+\delta}X\sum_{j=1}^2\int\limits_{-\Delta}^\Delta\big|S_{2, 2}(\lambda_jt)\big|^2\,dt\ll X^{\frac{50-7\gamma_2}{29}+\delta}\,.
\end{equation}

\textbf{Case 2.} $k=3$.

For Lemma \ref{Shapiroasymp}, Lemma \ref{S3asymptotic} and Lemma \ref{intS2}, we derive
\begin{equation}\label{J'5est2}
J'_5\ll X^{\frac{143-48\gamma_3}{288}+\delta}X\sum_{j=1}^2\int\limits_{-\Delta}^\Delta\big|S_{2, 3}(\lambda_jt)\big|^2\,dt\ll X^{\frac{431-48\gamma_3}{288}+\delta}\,.
\end{equation}

\textbf{Case 3.} $k=4$.

Applying Lemma \ref{Shapiroasymp}, Lemma \ref{S4asymptotic} and Lemma \ref{intS2}, we obtain
\begin{equation}\label{J'5est3}
J'_5\ll X^{\frac{149-50\gamma_4}{398}+\delta}X\sum_{j=1}^2\int\limits_{-\Delta}^\Delta\big|S_{2, 4}(\lambda_jt)\big|^2\,dt\ll X^{\frac{547-50\gamma_4}{398}+\delta}\,.
\end{equation}
Next, we estimate $J''_5$. Now \eqref{Delta}, Cauchy's inequality, Lemma \ref{Shapiroasymp}, Lemma \ref{Languasco} and Lemma \ref{intS2} imply
\begin{align}\label{J''5est}
J''_5&\ll X^\frac{3}{2}\int\limits_{-\Delta}^\Delta\big|S_{2, k}(\lambda_1t)\big|\big|\Sigma_k(\lambda_5t)-U_k(\lambda_5t)\big|\,dt \nonumber\\            
&\ll X^\frac{3}{2}\Bigg(\int\limits_{-\Delta}^\Delta\big|S_{2, k}(\lambda_1t)\big|^2\,dt\Bigg)^\frac{1}{2}
\Bigg(\int\limits_{-\Delta}^\Delta\big|\Sigma_k(\lambda_5t)-U_k(\lambda_5t)\big|^2\,dt\Bigg)^\frac{1}{2}\nonumber\\ 
&\ll \big(X\log X\big)^\frac{3}{2}\Bigg(\frac{X^{\frac{2}{k}-1}}{e^{(\log X)^{1/4}}}\Bigg)^\frac{1}{2}\ll\frac{X^{1+\frac{1}{k}}}{e^{(\log X)^{1/5}}}\,.
\end{align}
Finally, we estimate $J'''_5$. By \eqref{Delta}, \eqref{I-U}, Lemma \ref{Shapiroasymp} and Lemma \ref{intS2}, it follows that
\begin{align}\label{J'''5est}
J'''_5&\ll (1+\Delta X)X\sum_{j=1}^2\int\limits_{-\Delta}^\Delta\big|S_{2, k}(\lambda_jt)\big|^2\,dt\ll\Delta X^2\log^3X\,.
\end{align}

\subsection{Lower bound for $\mathbf{J}$}
\indent

For the integral defined by \eqref{JX}, we have
\begin{equation}\label{JXest}
J=D(X)+\Phi\,,
\end{equation}
where
\begin{equation*}
D(X)=\gamma^5\int\limits_{-\infty}^{\infty}\Theta(t)I_2(\lambda_1t)I_2(\lambda_2t)I_2(\lambda_3t)I_2(\lambda_4t)I_k(\lambda_5t)e(\eta t)\,dt
\end{equation*}
and
\begin{equation}\label{Phi}
\Phi\ll\int\limits_{\Delta}^{\infty }|\Theta(t)|\big|I_2(\lambda_1t)I_2(\lambda_2t)I_2(\lambda_3t)I_2(\lambda_4t)I_k(\lambda_5t)\big|\,dt\,.
\end{equation} 
Arguing as in (\cite{Dimitrov2015}, Lemma 4), we deduce that if $\lambda_0$ is sufficiently small, then
\begin{equation}\label{BXest}
D(X)\gg\varepsilon_k X^{1+\frac{1}{k}}\,.
\end{equation}
Using \eqref{Phi}, Lemma \ref{Fourier} and Lemma \ref{Ikest}, we deduce
\begin{equation}\label{Phiest}
\Phi\ll\frac{\varepsilon_k X^{\frac{1}{k}-3}}{\Delta^4}\,.
\end{equation}

\subsection{Estimation of $\mathbf{A_k(X)}$}
\indent

Taking into account \eqref{Delta}, \eqref{Gamma1-JX}, \eqref{J1est} -- \eqref{JXest}, \eqref{BXest} and \eqref{Phiest}, we derive 
\begin{equation}\label{Akest}
A_k(X)\gg\varepsilon_k X^{1+\frac{1}{k}}\,, \quad k=2, 3, 4\,.
\end{equation}

\section{Upper bound for $\mathbf{B_k(X)}$}
\indent
Suppose that
\begin{equation}\label{taq}
\bigg|t-\frac{a}{q}\bigg|\leq\frac{1}{q^2}\,,\quad (a, q)=1
\end{equation}
with
\begin{equation}\label{Intq}
q\in\left[X^{\frac{2}{29}}, X^{\frac{27}{29}}\right]\,.
\end{equation}
Then \eqref{Sigma}, \eqref{taq}, \eqref{Intq} and  Lemma \ref{Expsumest} lead to
\begin{equation}\label{Sigmaest}
\Sigma_2(t)\ll X^{\frac{14}{29}+\delta}\,.
\end{equation}
Now \eqref{Sigmaest} and Lemma \ref{S2asymptotic} give us
\begin{equation}\label{Salphaest}
S_{2, 2}(t)\ll X^{\frac{21-7\gamma_2}{29}+\delta}\,.
\end{equation}
Put
\begin{equation}\label{mathfrakS}
\mathfrak{S}_k(t,X)=\min\Big\{\big|S_{2, k}(\lambda_{1}t)\big|,\big|S_{2, k}(\lambda_2 t)\big|\Big\}\,.
\end{equation}
\begin{lemma}\label{mathfrakS2est} Let $\frac{71}{72}<\gamma_2<1$ and 
\begin{equation}\label{tdeltaH}
\Delta\leq|t|\leq H_2\,.
\end{equation}
Then there exists a sequence of real numbers $X_1,\,X_2,\ldots \to \infty $ such that
\begin{equation*}
\mathfrak{S}_2(t,X_j)\ll X_j^{\frac{21-7\gamma_2}{29}+\delta}\,,\quad j=1,2,\dots\,.
\end{equation*}
\end{lemma}
\begin{proof}
We aim to show that there exists a sequence $X_1, X_2, \ldots \to \infty$ such that, for each $j = 1,2,\ldots$, at least one of the quantities $\lambda_1 t$ or $\lambda_2 t$, 
where $t$ satisfies \eqref{tdeltaH}, admits a rational approximation with denominator satisfying \eqref{Intq}. The result then follows from \eqref{Salphaest} and \eqref{mathfrakS}.
Since $\lambda_1, \lambda_2, \lambda_3, \lambda_4, \lambda_5$  are not all of the same sign, one can assume that $\lambda_1>0$, $\lambda_2>0$, $\lambda_3>0$, $\lambda_4>0$  and $\lambda_5<0$.
We note that there exist $a_1,\,q_1\in \mathbb{Z}$, such that
\begin{equation}\label{lambda1a1q1}
\bigg|\lambda_1t-\frac{a_1}{q_1}\bigg|<\frac{1}{q_1q_0^2}\,, \quad\quad (a_1,\,q_1)=1,\quad\quad 1\leq q_1\leq q_0^2,\quad\quad a_1\ne 0\,.
\end{equation}
From Dirichlet's approximation theorem, it follows that there exist integers $a_1$ and $q_1$ satisfying the first three conditions. If $a_1=0$ then
\begin{equation*}
|\lambda_1t|< \frac{1}{q_1q_0^2}
\end{equation*}
and \eqref{tdeltaH} yield
\begin{equation*}
\lambda_1\Delta< \lambda_1|t|< \frac{1}{q_0^2}\,,\quad\quad
q_0^2< \frac{1}{\lambda_1\Delta}\,.
\end{equation*}
The last inequality,  (\ref{X}) and (\ref{Delta}) imply
\begin{equation*}
X^{\frac{27}{29}}<\frac{X^{\frac{27}{29}}}{\lambda_1\log X}\,,
\end{equation*}
which is impossible for large $X$. Consequently $a_1\ne 0$. Similarly, there exist $a_2,\,q_2\in \mathbb{Z}$, such that
\begin{equation}\label{lambda2a2q2}
\bigg|\lambda_2t-\frac{a_2}{q_2}\bigg|<\frac{1}{q_2q_0^2}\,, \quad\quad (a_2,\,q_2)=1,\quad\quad 1\leq q_2\leq q_0^2,\quad\quad a_2\ne 0\,.
\end{equation}
It suffices that $q_i \in \left[X^{\frac{2}{29}}, X^{\frac{27}{29}}\right]$ for some $i \in {1,2}$ in order to complete the proof.
From \eqref{X}, \eqref{lambda1a1q1} and \eqref{lambda2a2q2}, we have
\begin{equation*}
q_i\le X^{\frac{27}{29}}=q_0^2\,,\quad i=1,2\,.
\end{equation*}
It remains to prove that $q_i<X^{\frac{2}{29}}$ for $i = 1,2$ cannot occur. Assume that
\begin{equation}\label{impossible}
q_i<X^{\frac{2}{29}}\,,\quad i=1,2\,.
\end{equation}
From \eqref{varepsilon2}, \eqref{H}, \eqref{tdeltaH} -- \eqref{impossible}, we obtain
\begin{align}
  & 1\le |a_i|<\frac{1}{q_0^2}+q_i\lambda_i|t|< \frac{1}{q_0^2}+q_i\lambda_i H_2\,,\nonumber\\
\label{ai}
& 1\le |a_i|<\frac{1}{q_0^2}+\lambda_iX^{\frac{72\gamma_2-67}{58}-\theta}\,,\quad i=1,\,2\,.
\end{align}
On the other hand
\begin{equation}\label{lambda12}
  \frac{\lambda_1}{\lambda_2}=\frac{\lambda_1t}{\lambda_2t}=
  \frac{\frac{a_1}{q_1}+\bigg(\lambda_1t-\frac{a_1}{q_1}\bigg)}{\frac{a_2}{q_2}+\bigg(\lambda_2t-\frac{a_2}{q_2}\bigg)}=
  \frac{a_1q_2}{a_2q_1}\cdot\frac{1+\mathfrak{X}_1}{1+\mathfrak{X}_2}\,,
\end{equation}
where
\begin{equation}\label{mathfrakX}
\mathfrak{X}_i=\dfrac{q_i}{a_i}\bigg(\lambda_it-\dfrac{a_i}{q_i}\bigg)\,,\; i=1,\,2.
\end{equation}
Taking into account \eqref{lambda1a1q1}, \eqref{lambda2a2q2}, \eqref{lambda12} and \eqref{mathfrakX}, we get
\begin{align}
  &|\mathfrak{X}_i|< \frac{q_i}{|a_i|}\cdot \frac{1}{q_iq_0^2}=\frac{1}{|a_i|q_0^2}\le \frac{1}{q_0^2}\,,\quad i=1,2\,,\nonumber\\
\label{lambd12}
  &\frac{\lambda_1}{\lambda_2}=\frac{a_1q_2}{a_2q_1}\cdot
  \frac{ 1+\mathcal{O}\bigg(\frac{1}{q_0^2}\bigg)}{ 1+\mathcal{O}\bigg(\frac{1}{q_0^2}\bigg)}=
 \frac{a_1q_2}{a_2q_1}\bigg(1+\mathcal{O}\bigg(\frac{1}{q_0^2}\bigg)\bigg)\,.\notag
\end{align}
Therefore
\begin{equation*}
\frac{a_1q_2}{a_2q_1}=\mathcal{O}(1)
\end{equation*}
and
\begin{equation}\label{lambd12new}
\frac{\lambda_1}{\lambda_2}=\frac{a_1q_2}{a_2q_1}+\mathcal{O}\bigg(\frac{1}{q_0^2}\bigg)\,.
\end{equation}
Consequently, both $\displaystyle \frac{a_0}{q_0}$ and $\displaystyle  \frac{a_1q_2}{a_2q_1}$ are rational approximations to $\displaystyle  \frac{\lambda_1}{\lambda_2}$.
Now \eqref{X}, \eqref{lambda1a1q1}, \eqref{impossible} and inequality \eqref{ai} with $i=2$ lead to
\begin{equation}\label{a2q1}
|a_2|q_1<1+\lambda_2X^{\frac{72\gamma_2-63}{58}-\theta}<\frac{q_0}{\log X}\,.
\end{equation}
Thus $|a_2|q_1\ne q_0$  and  $\displaystyle  \frac{a_0}{q_0}\neq\frac{a_1q_2}{a_2q_1}$.
Now \eqref{a2q1} gives
\begin{equation}\label{contradiction}
  \bigg|\frac{a_0}{q_0}-\frac{a_1q_2}{a_2q_1}\bigg|=
  \frac{|a_0 a_2q_1-a_1q_2q_0|}{|a_2|q_1q_0}\ge \frac{1}{|a_2|q_1q_0}>\frac{\log X}{q_0^2}\,.
\end{equation}
On the other hand, from \eqref{lambda12a0q0} and \eqref{lambd12new}, we derive
\begin{equation*}
  \bigg|\frac{a_0}{q_0}-\frac{a_1q_2}{a_2q_1}\bigg|\le \bigg|\frac{a_0}{q_0}-\frac{\lambda_1}{\lambda_2}\bigg|+
  \bigg|\frac{\lambda_1}{\lambda_2}-\frac{a_1q_2}{a_2q_1}\bigg|\ll \frac{1}{q_0^2}\,,
\end{equation*}
which contradicts \eqref{contradiction}.
This rejects the assumption \eqref{impossible}.
Let $q_0^{(1)},\,q_0^{(2)},\,\ldots$ be an infinite sequence of values of $q_0$, satisfying \eqref{lambda12a0q0}.
Then using \eqref{X} one gets an infinite sequence $X_1,\,X_2,\,\ldots $ of values of $X$, such that
at least one of the numbers $\lambda_{1}t$ and $\lambda_{2}t$ can be approximated by
rational numbers with denominators, satisfying \eqref{Intq}.
This completes the proof of the lemma.
\end{proof}

As a byproduct of Lemma \ref{mathfrakS2est}, we obtain the following two corollaries.
\begin{corollary}\label{mathfrakS3est} Let $\frac{129}{130}<\gamma_3<1$ and $\Delta\leq|t|\leq H_3$.
Then there exists a sequence of real numbers $X_1,\,X_2,\ldots \to \infty $ such that
\begin{equation*}
\mathfrak{S}_3(t,X_j)\ll X_j^{\frac{21-7\gamma_3}{29}+\delta}\,,\quad j=1,2,\dots\,.
\end{equation*}
\end{corollary}

\begin{corollary}\label{mathfrakS4est} Let $\frac{245}{246}<\gamma_4<1$ and $\Delta\leq|t|\leq H_4$.
Then there exists a sequence of real numbers $X_1,\,X_2,\ldots \to \infty $ such that
\begin{equation*}
\mathfrak{S}_4(t,X_j)\ll X_j^{\frac{21-7\gamma_4}{29}+\delta}\,,\quad j=1,2,\dots\,.
\end{equation*}
\end{corollary}
We are now in a good position to estimate the integral $B_k(X_j)$ in three cases..

\smallskip

\textbf{Case A.} $k=2$.

Bearing in mind \eqref{Bk}, \eqref{mathfrakS}, Lemma \ref{Fourier} and Lemma \ref{mathfrakS2est}, we deduce

\begin{align}\label{B2est1}
B_2(X_j)&\ll\varepsilon_2\int\limits_{\Delta\leq|t|\leq H_2}\mathfrak{S}_2(t, X_j)\big|S_{2, 2}(\lambda_2t)S_{2, 2}(\lambda_3t)S_{2, 2}(\lambda_4t)S_{2, 2}(\lambda_5t)\big|\,dt\nonumber\\
&+\varepsilon_2\int\limits_{\Delta\leq|t|\leq H_2}\mathfrak{S}_2(t, X_j)\big|S_{2, 2}(\lambda_1t)S_{2, 2}(\lambda_3t)S_{2, 2}(\lambda_4t)S_{2, 2}(\lambda_5t)\big|\,dt\nonumber\\
&\ll\varepsilon_2 X_j^{\frac{21-7\gamma_2}{29}+\delta}\big(\Psi_1+\Psi_2\big)\,,
\end{align}
where
\begin{align}
\label{Psi1}
&\Psi_1=\int\limits_{\Delta}^{H_2}\big|S_{2, 2}(\lambda_2t)S_{2, 2}(\lambda_3t)S_{2, 2}(\lambda_4t)S_{2, 2}(\lambda_5t)\big|\,dt\,,\\
&\Psi_2=\int\limits_{\Delta}^{H_2}\big|S_{2, 2}(\lambda_1t)S_{2, 2}(\lambda_3t)S_{2, 2}(\lambda_4t)S_{2, 2}(\lambda_5t)\big|\,dt\,.\nonumber
\end{align}
We estimate only $\Psi_1$ and the estimation of $\Psi_2$ proceeds in the same way. Now \eqref{Psi1} and Lemma \ref{intS4} imply
\begin{equation}\label{Psi1est1}
\Psi_1\ll\sum_{j=2}^5\int\limits_{\Delta}^{H_2}\big|S_{2, 2}(\lambda_jt)\big|^4\,dt\ll H_2X^{2-\gamma_2+\delta}\,.
\end{equation}
Combining \eqref{varepsilon2}, \eqref{H}, \eqref{B2est1} and \eqref{Psi1est1}, we get
\begin{equation}\label{B2est2}
B_2(X_j)\ll  X_j^{\frac{21-7\gamma_2}{29}+\delta} X^{2-\gamma_2+\delta}=X_j^{\frac{79-36\gamma_2}{29}+\delta}\ll\frac{\varepsilon_2 X_j^\frac{3}{2}}{\log X_j}\,.
\end{equation}

\textbf{Case B.} $k=3$.

Using \eqref{Bk}, \eqref{mathfrakS}, Lemma \ref{Fourier} and Corollary \ref{mathfrakS3est}, we derive
\begin{align}\label{B3est1}
B_3(X_j)&\ll\varepsilon_3\int\limits_{\Delta\leq|t|\leq H_3}\mathfrak{S}_3(t, X_j)^\frac{1}{2}\big|S_{2, 3}(\lambda_1t)\big|^\frac{1}{2}\big|S_{2, 3}(\lambda_2t)S_{2, 3}(\lambda_3t)S_{2, 3}(\lambda_4t)S_{3, 3}(\lambda_5t)\big|\,dt\nonumber\\
&+\varepsilon_3\int\limits_{\Delta\leq|t|\leq H_3}\mathfrak{S}_3(t, X_j)^\frac{1}{2}\big|S_{2, 3}(\lambda_2t)\big|^\frac{1}{2}\big|S_{2, 3}(\lambda_1t)S_{2, 3}(\lambda_3t)S_{2, 3}(\lambda_4t)S_{3, 3}(\lambda_5t)\big|\,dt\nonumber\\
&\ll\varepsilon_3 X_j^{\frac{21-7\gamma_3}{58}+\delta}\big(\Psi'_1+\Psi'_2\big)\,,
\end{align}
where
\begin{align}
\label{Psi'1}
&\Psi'_1=\int\limits_{\Delta}^{H_3}\big|S_{2, 3}(\lambda_1t)\big|^\frac{1}{2}\big|S_{2, 3}(\lambda_2t)S_{2, 3}(\lambda_3t)S_{2, 3}(\lambda_4t)S_{3, 3}(\lambda_5t)\big|\,dt\,,\\
&\Psi'_2=\int\limits_{\Delta}^{H_3}\big|S_{2, 3}(\lambda_2t)\big|^\frac{1}{2}\big|S_{2, 3}(\lambda_1t)S_{2, 3}(\lambda_3t)S_{2, 3}(\lambda_4t)S_{3, 3}(\lambda_5t)\big|\,dt\,.\nonumber
\end{align}
We estimate only $\Psi'_1$, since the estimate for $\Psi'_2$ is obtained in the same way. Now \eqref{Psi'1}, Hölder's inequality, Lemma \ref{intS4} and Lemma \ref{intS8} yield
\begin{align}\label{Psi'1est1}
\Psi'_1&\ll\sum_{j=2}^4\int\limits_{\Delta}^{H_3}\big|S_{2, 3}(\lambda_1t)\big|^\frac{1}{2}\big|S_{2, 3}(\lambda_jt)\big|^3\big|S_{3, 3}(\lambda_5t)\big|\,dt\nonumber\\
&\ll\sum_{j=2}^4\Bigg(\int\limits_{\Delta}^{H_3}\big|S_{2, 3}(\lambda_jt)\big|^4\,dt\Bigg)^\frac{3}{4}\Bigg(\int\limits_{\Delta}^{H_3}\big|S_{2, 3}(\lambda_1t)\big|^4\,dt\Bigg)^\frac{1}{8}
\Bigg(\int\limits_{\Delta}^{H_3}\big|S_{3, 3}(\lambda_5t)\big|^8\,dt\Bigg)^\frac{1}{8}\nonumber\\
&\ll H_3\Big(X^{2-\gamma_3+\delta}\Big)^{\frac{3}{4}+\frac{1}{8}} \Big(X^{\frac{8}{3}-\gamma_3+\delta}\Big)^\frac{1}{8}\nonumber\\
&\ll H_3X_j^{\frac{25-12\gamma_3}{12}+\delta}\,.
\end{align}
In view of \eqref{varepsilon3}, \eqref{H}, \eqref{B3est1} and \eqref{Psi'1est1}, we obtain
\begin{equation}\label{B3est2}
B_3(X_j)\ll  X_j^{\frac{21-7\gamma_3}{58}+\delta} X_j^{\frac{25-12\gamma_3}{12}+\delta}=X_j^{\frac{851-390\gamma_3}{348}+\delta}\ll\frac{\varepsilon_3 X_j^\frac{4}{3}}{\log X_j}\,.
\end{equation}

\textbf{Case C.} $k=4$.

By \eqref{Bk}, \eqref{mathfrakS}, Lemma \ref{Fourier} and Corollary \ref{mathfrakS4est}, we  deduce

\begin{align}\label{B4est1}
B_4(X_j)&\ll\varepsilon_4\int\limits_{\Delta\leq|t|\leq H_4}\mathfrak{S}_4(t, X_j)^\frac{1}{4}\big|S_{2, 4}(\lambda_1t)\big|^\frac{3}{4}\big|S_{2, 4}(\lambda_2t)S_{2, 4}(\lambda_3t)S_{2, 4}(\lambda_4t)S_{4, 4}(\lambda_5t)\big|\,dt\nonumber\\
&+\varepsilon_4\int\limits_{\Delta\leq|t|\leq H_4}\mathfrak{S}_4(t, X_j)^\frac{1}{4}\big|S_{2, 4}(\lambda_2t)\big|^\frac{3}{4}\big|S_{2, 4}(\lambda_1t)S_{2, 4}(\lambda_3t)S_{2, 4}(\lambda_4t)S_{4, 4}(\lambda_5t)\big|\,dt\nonumber\\
&\ll\varepsilon_4 X_j^{\frac{21-7\gamma_4}{116}+\delta}\big(\Psi''_1+\Psi''_2\big)\,,
\end{align}
where
\begin{align}
\label{Psi''1}
&\Psi''_1=\int\limits_{\Delta}^{H_4}\big|S_{2, 4}(\lambda_1t)\big|^\frac{3}{4}\big|S_{2, 4}(\lambda_2t)S_{2, 4}(\lambda_3t)S_{2, 4}(\lambda_4t)S_{4, 4}(\lambda_5t)\big|\,dt\,,\\
&\Psi''_2=\int\limits_{\Delta}^{H_4}\big|S_{2, 4}(\lambda_2t)\big|^\frac{3}{4}\big|S_{2, 4}(\lambda_1t)S_{2, 4}(\lambda_3t)S_{2, 4}(\lambda_4t)S_{4, 4}(\lambda_5t)\big|\,dt\,.\nonumber
\end{align}
Only the estimate of $\Psi''_1$ is provided, as the case of $\Psi''_2$ is handled similarly. Now \eqref{Psi''1}, Hölder's inequality, Lemma \ref{intS4} and Lemma \ref{intS16} give us
\begin{align}\label{Psi''1est1}
\Psi''_1&\ll\sum_{j=2}^4\int\limits_{\Delta}^{H_4}\big|S_{2, 4}(\lambda_1t)\big|^\frac{3}{4}\big|S_{2, 4}(\lambda_jt)\big|^3\big|S_{4, 4}(\lambda_5t)\big|\,dt\nonumber\\
&\ll\sum_{j=2}^4\Bigg(\int\limits_{\Delta}^{H_4}\big|S_{2, 4}(\lambda_jt)\big|^4\,dt\Bigg)^\frac{3}{4}\Bigg(\int\limits_{\Delta}^{H_4}\big|S_{2, 4}(\lambda_1t)\big|^4\,dt\Bigg)^\frac{3}{16}
\Bigg(\int\limits_{\Delta}^{H_4}\big|S_{4, 4}(\lambda_5t)\big|^{16}\,dt\Bigg)^\frac{1}{16}\nonumber\\
&\ll H_4 \Big(X^{2-\gamma_4+\delta}\Big)^{\frac{3}{4}+\frac{3}{16}}\Big(X^{4-\gamma_4+\delta}\Big)^\frac{1}{16}\nonumber\\
&\ll H_4 X_j^{\frac{17-8\gamma_4}{8}+\delta}\,.
\end{align}
Using \eqref{varepsilon4}, \eqref{H}, \eqref{B4est1} and \eqref{Psi''1est1}, we get
\begin{equation}\label{B4est2}
B_4(X_j)\ll  X_j^{\frac{21-7\gamma_4}{116}+\delta} X_j^{\frac{17-8\gamma_4}{8}+\delta}=X_j^{\frac{535-246\gamma_4}{232}+\delta}\ll\frac{\varepsilon_4 X_j^\frac{5}{4}}{\log X_j}\,.
\end{equation}

\section{Upper bound for $\mathbf{C_k(X)}$}
\indent

From \eqref{Sk}, \eqref{Ck}, Lemma \ref{Fourier} and Lemma \ref{Shapiroasymp}, we derive
\begin{equation}\label{Gamma3est1}
C_k(X)\ll X^5\int\limits_{H_k}^{\infty}\frac{1}{t}\bigg(\frac{l}{2\pi t\varepsilon_k/8}\bigg)^l \,dt=\frac{X^5}{l}\bigg(\frac{4l}{\pi\varepsilon_k H_k}\bigg)^l\,.
\end{equation}
Choosing $l=[\log X]$ from \eqref{H} and \eqref{Gamma3est1}, we obtain
\begin{equation}\label{Ckest}
C_k(X)\ll1\,, \quad k=2, 3, 4\,.
\end{equation}

\section{Proof of Theorem 1}
\indent

Summarizing  \eqref{varepsilon2}, \eqref{Gammadecomp}, \eqref{Akest}, \eqref{B2est2} and \eqref{Ckest}, we find
\begin{equation*}
\Gamma_2(X_j)\gg\varepsilon_2 X_j^\frac{3}{2}=X_j^{\frac{79-36\gamma_2}{29}+\theta}\,.
\end{equation*}
The last estimation implies
\begin{equation}\label{Lowerbound1}
\Gamma_2(X_j) \rightarrow\infty \quad \mbox{ as } \quad X_j\rightarrow\infty\,.
\end{equation}
Bearing in mind \eqref{Gamma} and \eqref{Lowerbound1} we establish Theorem \ref{Theorem1}.

\section{Proof of Theorem 2}
\indent

From \eqref{varepsilon3}, \eqref{Gammadecomp}, \eqref{Akest}, \eqref{B3est2} and \eqref{Ckest}, it follows that
\begin{equation*}
\Gamma_3(X_j)\gg\varepsilon_3 X_j^\frac{4}{3}=X_j^{\frac{851-390\gamma_3}{348}+\theta}\,.
\end{equation*}
The last estimation implies
\begin{equation}\label{Lowerbound2}
\Gamma_3(X_j) \rightarrow\infty \quad \mbox{ as } \quad X_j\rightarrow\infty\,.
\end{equation}
Taking into account \eqref{Gamma} and \eqref{Lowerbound2} we establish Theorem \ref{Theorem2}.

\section{Proof of Theorem 3}
\indent

By \eqref{varepsilon4}, \eqref{Gammadecomp}, \eqref{Akest}, \eqref{B4est2} and \eqref{Ckest}, we deduce
\begin{equation*}
\Gamma_4(X_j)\gg\varepsilon_4 X_j^\frac{5}{4}=X_j^{\frac{535-246\gamma_4}{232}+\theta}\,.
\end{equation*}
The last estimation implies
\begin{equation}\label{Lowerbound3}
\Gamma_4(X_j) \rightarrow\infty \quad \mbox{ as } \quad X_j\rightarrow\infty\,.
\end{equation}
In view of \eqref{Gamma} and \eqref{Lowerbound3} we establish Theorem \ref{Theorem3}.

\vskip30pt
\footnotesize
\begin{flushleft}
S. I. Dimitrov\\
\quad\\
Faculty of Applied Mathematics and Informatics\\
Technical University of Sofia \\
Blvd. St. Kliment Ohridski 8 \\
Sofia 1000, Bulgaria\\
e-mail: sdimitrov@tu-sofia.bg\\
\end{flushleft}

\begin{flushleft}
Department of Bioinformatics and Mathematical Modelling\\
Institute of Biophysics and Biomedical Engineering\\
Bulgarian Academy of Sciences\\
Acad. G. Bonchev Str. Bl. 105, Sofia 1113, Bulgaria \\
e-mail: xyzstoyan@gmail.com\\
\end{flushleft}

\end{document}